\documentclass[12pt]{amsart}
\usepackage{amsmath,amssymb,amsthm}
\usepackage[margin=1.1in]{geometry}
\usepackage[colorlinks=true,linkcolor=black,citecolor=black]{hyperref}
\newcommand{\R}{\mathbb{R}}
\newcommand{\M}{\mathbb{M}^3(c)}
\newcommand{\sn}{\operatorname{sn}}
\newcommand{\cs}{\operatorname{cs}}
\newcommand{\ta}{\operatorname{ta}}
\newcommand{\ct}{\operatorname{ct}}
\newcommand{\dA}{\,dA}

\newtheorem{theorem}{Theorem}

\newtheorem{lemma}{Lemma}[section]
\newtheorem{proposition}[lemma]{Proposition}
\newtheorem{corollary}[lemma]{Corollary}
\theoremstyle{remark}
\newtheorem{remark}[lemma]{Remark}
\numberwithin{equation}{section}

\title[CMC disks with circular boundary in $\mathbb S^3$ and $\mathbb H^3$]{Constant mean curvature disks with circular boundary in the $3$-sphere and in hyperbolic space}
\author{Rafael L\'opez}
\address{Department of Geometry and Topology. University of Granada. 18071 Granada, Spain}
\email{rcamino@ugr.es}

\begin{document}

\begin{abstract}
We prove that a compact immersed disk of constant mean curvature in $\mathbb S^3$ or in
$\mathbb H^3$ whose boundary is mapped diffeomorphically onto a circle is an embedded
totally umbilical disk.  
\end{abstract}

\maketitle

\section{Introduction}

By a theorem of H.~Hopf, an immersed sphere of constant mean curvature in $\R^3$ is a round
sphere \cite{Hop83}. For surfaces with boundary the simplest analogue asks whether a compact
immersed disk of constant mean curvature bounded by a circle is a planar disk or a spherical
cap. The question was posed by R.~Gulliver and R.~Kusner in 1984 \cite{Bro86} and has recently
been answered in the affirmative, independently and with different methods, by
J.~M.~Espinar  \cite{Esp26} and by D.~Maximo and I.~Nunes \cite{MN26}. See
\cite{Lop13,Lop25} for the history of the problem and related results.

The same question makes sense in the other simply connected space forms, the sphere
$\mathbb S^3$ and hyperbolic space $\mathbb H^3$. There the natural candidates are the disks
bounded by the circle in totally umbilical surfaces. In this paper we show that they are the
only ones. For stable disks this was proved in 
\cite{ALP99}.
We write $\M$ for $\mathbb S^3$ if $c=1$ and for $\mathbb H^3$ if $c=-1$.

\begin{theorem}\label{thm:A}
Let $c\in\{1,-1\}$, let $M$ be a smooth compact surface diffeomorphic to the closed disk
and let $x\colon M\to\M$ be a smooth immersion with constant mean curvature $H$ such that
$x|_{\partial M}$ is a diffeomorphism onto a circle.
Then $x$ is an embedding onto a closed disk bounded by the circle in a totally umbilical
surface of $\M$ with mean curvature $H$.  
\end{theorem}

In $\mathbb H^3$ with $|H|\le1$, Theorem~\ref{thm:A} was proved by the author in \cite{lo},
for compact surfaces of any genus. The new cases are therefore $\mathbb S^3$, for every
value of $H$, and $\mathbb H^3$ with $|H|>1$.

We follow the strategy of Maximo and Nunes \cite{MN26}, based on the force form, integral
identities and Hurwitz's isoperimetric inequality \cite{Hur01}. In $\R^3$ the force form comes
from the translations. In $\M$ their role is played by the rotations of $\mathbb S^3$, or the
boosts of $\mathbb H^3$, that move the centre of the circle. The main new point is the use of a
second family of Killing fields, the rotations of $\M$ about the geodesics through the centre
of the circle that lie in a totally geodesic surface containing it. Along the boundary both
families give closed planar curves, and their areas and energies are related by an identity
(Proposition~\ref{prop:key}) in which the second family enters with the sign $c$. For $c=1$ this
identity gives the result at once. For $c=-1$ the signs are opposite, and a Fourier estimate
together with the flux formula settles the case $|H|>1$. 

We first prove the following
boundary result.

\begin{theorem}\label{thm:B}
Under the hypotheses of Theorem~\ref{thm:A}, assume that $c=1$, or that $c=-1$ and
$|H|>1$. Then, along $\partial M$, $x(M)$ meets at a constant angle a totally geodesic
surface containing the boundary circle, and the shape operator equals $H\,\mathrm{Id}$ at every
point of $\partial M$.
\end{theorem}

From Theorem~\ref{thm:B}, the Hopf differential vanishes along the boundary, and as in
\cite{MN26} it vanishes identically. Embeddedness follows from Gauss--Bonnet and a degree
argument.

The paper is organized as follows. Section~\ref{sec:prelim} contains the notation, the
closed one-forms associated with Killing fields, the flux formulas and the Hurwitz type
inequality. Section~\ref{sec:B} proves Theorem~\ref{thm:B}, and Section~\ref{sec:A} proves
Theorem~\ref{thm:A}.

%%%%
\section{Preliminaries}\label{sec:prelim}

Let $c\in\{1,-1\}$. On $\R^4$ consider $\langle X,Y\rangle=cX_0Y_0+X_1Y_1+X_2Y_2+X_3Y_3$
and let $\M=\{X:\langle X,X\rangle=c\}$, taking the sheet $X_0>0$ when $c=-1$. Let
$E_0,\dots,E_3$ be the standard basis and $G=\mathrm{diag}(c,1,1,1)$. We write
\[
\sn r=\begin{cases}\sin r\\ \sinh r\end{cases},\quad
\cs r=\begin{cases}\cos r\\ \cosh r\end{cases},\quad
\ta=\frac{\sn}{\cs},\quad \ct=\frac{\cs}{\sn},\qquad \cs^2+c\,\sn^2=1 .
\]
For $X\in\M$ and $u,v\in T_X\M$ define $u\times v\in T_X\M$ by
$\langle u\times v,w\rangle=\det(X,u,v,w)$, where $\det$ is the determinant of the
coordinate matrix. If $F=(X,f_1,f_2,f_3)$ with $(f_i)$ orthonormal in $T_X\M$, then
$F^{T}GF=G$, so $\det F=\pm1$. Hence $\det(X,\cdot,\cdot,\cdot)$ is a unit volume form
on $T_X\M$, and $\times$ is the usual cross product of the oriented Euclidean space
$T_X\M$.

For an immersion $x\colon M\to\M$ with unit normal $N$ (so $\langle N,x\rangle=0$) the
ambient derivative of $N$ is tangent to $M$, and $dN=-dx\circ A$, with $A$ the shape operator
and $H=\frac12\operatorname{tr}A$. A positive orthonormal
frame $(e_1,e_2)$ of $TM$ is one with $e_1\times e_2=N$. Then 
$$N\times e_1=e_2,\qquad N\times e_2=-e_1,\qquad \det(x,e_1,e_2,N)=1.$$
 Throughout, $(e_1,e_2)$ denotes a positive
orthonormal frame of $TM$, identified with its image under $dx$.

Every circle of $\mathbb S^3$ has radius $\rho\in(0,\pi/2]$ about one of its two centres,
and every circle of $\mathbb H^3$ has some radius $\rho>0$. After an isometry we assume
that $x(\partial M)$ is the circle
\begin{equation*}
\begin{split}
x(p(\theta))&=\cs\rho\,E_0+\sn\rho\,e_r(\theta),\\
 e_r&=\cos\theta\,E_1+\sin\theta\,E_2,\\
  e_\theta&=-\sin\theta\,E_1+\cos\theta\,E_2,
\end{split}
\end{equation*}
for a diffeomorphism $p\colon\R/2\pi\mathbb Z\to\partial M$. Then
$\frac{d}{d\theta}x(p(\theta))=\sn\rho\,t$ with $t=e_\theta$, and $d\ell=\sn\rho\,d\theta$ is the length element of $\partial M$.
The circle lies in the totally geodesic surface $\M\cap\{X_3=0\}$, whose unit normal is
$E_3$. Along $\partial M$ the vectors
\[
\{\partial_r=-c\,\sn\rho\,E_0+\cs\rho\,e_r,\ t,\ E_3\}
\]
form a positive orthonormal basis of $T_x\M$. Let $\mu$ be the outward unit conormal and fix
$N$ by $N=t\times\mu$ along $\partial M$. Orient $M$ by $N$. Writing $N=s\,\partial_r+n\,E_3$,
we have
\begin{equation}\label{eq:Nmu}
  \mu=N\times t=-n\,\partial_r+s\,E_3,\qquad s^2+n^2=1 .
\end{equation}
The boundary orientation induced by $M$ is $-t$. Consequently,
for every smooth $F=(F_1,F_2)\colon M\to\R^2$, Stokes' theorem gives
\begin{equation}\label{eq:stokes}
\int_M dF_1\wedge dF_2=-\frac12\int_0^{2\pi}\det\big(F,\dfrac{d}{d\theta}F\big)\,d\theta ,
\end{equation}
where $F$ is evaluated at $p(\theta)$ on the right. We also use
\[
x_0:=X_0\circ x,\quad x_3:=\langle x,E_3\rangle,\quad S:=-\langle N,E_0\rangle,\quad N_3:=\langle N,E_3\rangle ,
\]
so that $\langle x,E_0\rangle=c\,x_0$ and $\langle N,E_0\rangle=-S$. Along $\partial M$ we have $x_0=\cs\rho$, $x_3=0$, $S=s\,\sn\rho$ and $N_3=n$.

Let $\nabla$ denote the Levi-Civita connection of $\M$. For vector fields tangent to $\M$,
$\nabla_UV$ is the tangential part of the derivative of $V$ in $\R^4$, which differs from it by
a multiple of $x$. Differentiating $t$ along the circle gives $\frac{d}{d\ell}t=-e_r/\sn\rho$.
Since $N$ and $\mu$ are orthogonal to $x$,
\begin{equation}\label{eq:curvatures}
A(t,t)=\langle -\frac{e_r}{\sn\rho},N\rangle=-s\,\ct\rho,\qquad
\big\langle \nabla_t t,\mu\big\rangle=n\,\ct\rho .
\end{equation}

We fix the constant mean curvature $H$ and set
\[
W:=c\,x-HN .
\]
For a linear map $B$ of $\R^4$ that is skew-adjoint for $\langle\cdot,\cdot\rangle$, define the one-form on $M$
\[
\omega_B(v)=c\,\big\langle BW,\,N\times dx(v)\big\rangle .
\]

This is the version in $\M$ of the force form of R.~L\'opez and S.~Montiel
\cite{LM95,LM96}.

\begin{lemma}\label{lem:closed}
If $H$ is constant, then $\omega_B$ is closed.
\end{lemma}

\begin{proof}
The map $a\wedge b\mapsto\langle Ba,b\rangle$ is well defined on $\Lambda^2\R^4$, and
$\omega_B$ is its composition with the $\Lambda^2\R^4$-valued form $\Psi=W\wedge\eta$,
$\eta(v)=N\times dx(v)$. It suffices to show 
$$d\Psi=d W\wedge\eta+W\wedge d\eta=0.$$
Let $(e_1,e_2)$ be a positive orthonormal frame, and write $Ae_j=a_{1j}e_1+a_{2j}e_2$,
so that $a_{12}=a_{21}$ and $a_{11}+a_{22}=2H$.

  Since $dW=dx\circ(c+HA)$ and $\eta(e_1)=e_2$, $\eta(e_2)=-e_1$,
\[
(dW\wedge\eta)(e_1,e_2)=-\sum_{i=1}^2\big((c+HA)e_i\big)\wedge e_i
=-H\big(a_{21}\,e_2\wedge e_1+a_{12}\,e_1\wedge e_2\big)=0 ,
\]
because $A$ is symmetric.

Also, we have $\eta=\star(x\wedge N\wedge dx)$, where $\star$ is
defined by $\langle\star(a\wedge b\wedge e),w\rangle=\det(a,b,e,w)$ for all $w\in\R^4$.
Differentiating, and using $d(dx)=0$ and $dN=-dx\circ A$,
\[
d\eta(e_1,e_2)=\mathbf a+\mathbf b,\]
where 
\[ 
\mathbf a:=\star\big(e_1\wedge N\wedge e_2-e_2\wedge N\wedge e_1\big),\qquad
\mathbf b:=\star\big(x\wedge Ae_2\wedge e_1-x\wedge Ae_1\wedge e_2\big).
\]
Since $e_2\wedge N\wedge e_1=-e_1\wedge N\wedge e_2$, we have
$\mathbf a=2\star(e_1\wedge N\wedge e_2)$. By definition of $\star$, the vector $\mathbf a$ is
orthogonal to $e_1$, $N$ and $e_2$, hence $\mathbf a=\lambda x$ for some $\lambda\in\R$.
Pairing with $x$, and noting that $(x,e_1,e_2,N)\mapsto(e_1,N,e_2,x)$ is an even permutation,
\[
c\lambda=\langle\mathbf a,x\rangle=2\det(e_1,N,e_2,x)=2\det(x,e_1,e_2,N)=2 .
\]
Thus $\mathbf a=2c\,x$. Next, 
$$x\wedge Ae_2\wedge e_1-x\wedge Ae_1\wedge e_2=-(a_{11}+a_{22})\,x\wedge e_1\wedge e_2
=-2H\,x\wedge e_1\wedge e_2,$$
 so $\mathbf b=-2H\star(x\wedge e_1\wedge e_2)$. This vector is
orthogonal to $x$, $e_1$ and $e_2$, hence a multiple of $N$, and
$\langle\mathbf b,N\rangle=-2H\det(x,e_1,e_2,N)=-2H$. Thus $\mathbf b=-2HN$.

Consequently $d\eta(e_1,e_2)=2(c\,x-HN)=2W$, and $W\wedge d\eta=2\,W\wedge W=0$. Together with $dW\wedge\eta=0$, we conclude that $\omega_B$ is closed. 
\end{proof}

\begin{lemma}\label{lem:det}
Let $(e_1,e_2)$ be a positive orthonormal frame of $TM$. For all $w,w'\in\R^4$,
\begin{align*}
\det(x,N,w,w')&=\langle w,e_1\rangle\langle w',e_2\rangle-\langle w,e_2\rangle\langle w',e_1\rangle,\\
\det(w,w',e_1,e_2)&=c\big(\langle w,x\rangle\langle w',N\rangle-\langle w,N\rangle\langle w',x\rangle\big).
\end{align*}
\end{lemma}

\begin{proof}
The vectors $x,N,e_1,e_2$ are mutually orthogonal, with $\langle x,x\rangle=c$ and
$\langle N,N\rangle=\langle e_i,e_i\rangle=1$. Hence every $w\in\R^4$ decomposes as
\[
w=c\langle w,x\rangle\,x+\langle w,N\rangle\,N+\langle w,e_1\rangle\,e_1+\langle w,e_2\rangle\,e_2 .
\]
Substitute this decomposition for $w$ and $w'$ and expand by multilinearity. In the first
determinant only the components along $e_1,e_2$ survive, and in the second only those along
$x,N$. Both identities then follow from $\det(x,N,e_1,e_2)=\det(x,e_1,e_2,N)=1$.
\end{proof}

\begin{lemma}\label{lem:wedge}
For skew-adjoint $B,B'$ and a positive orthonormal frame,
$$(\omega_B\wedge\omega_{B'})(e_1,e_2)=\det(x,N,BW,B'W).$$
\end{lemma}

\begin{proof}
Since $N\times e_1=e_2$ and $N\times e_2=-e_1$, we have
$\omega_B(e_1)=c\langle BW,e_2\rangle$ and $\omega_B(e_2)=-c\langle BW,e_1\rangle$, and
similarly for $B'$. Using $c^2=1$,
\[
(\omega_B\wedge\omega_{B'})(e_1,e_2)=\langle BW,e_1\rangle\langle B'W,e_2\rangle
-\langle BW,e_2\rangle\langle B'W,e_1\rangle ,
\]
which equals $\det(x,N,BW,B'W)$ by Lemma~\ref{lem:det}.
\end{proof}

We use the Killing fields
\begin{equation*}
\begin{split}
T_i(X)&=c\langle E_0,X\rangle E_i-c\langle E_i,X\rangle E_0\ (i=1,2,3),\\
R_1&=X_2E_3-X_3E_2,\\
 R_2&=X_3E_1-X_1E_3 .
 \end{split}
\end{equation*}
Since $M$ is a disk, Lemma~\ref{lem:closed} gives smooth functions
$P_1,P_2,P_3,Q_1,Q_2\colon M\to\R$, smooth up to the boundary and unique up to additive
constants, with
\[
dP_i=\omega_{T_i},\qquad dQ_1=\omega_{R_2},\qquad dQ_2=-\omega_{R_1}.
\]
We set
\[
q:=(P_1,P_2)\colon M\to\R^2,\qquad \tilde q:=(Q_1,Q_2)\colon M\to\R^2 .
\]
Along the boundary we use the same letters for the closed curves
\[
\theta\longmapsto q(p(\theta)),\qquad \theta\longmapsto\tilde q(p(\theta)),\qquad
\theta\longmapsto P_3(p(\theta)),
\]
and a prime denotes $d/d\theta$. The curve $q$ corresponds to the horizontal projection of
the force curve in \cite{MN26}.

\begin{lemma}\label{lem:formulas}
For a positive orthonormal frame,
\begin{equation}\label{eq:wedges}
\begin{split}
(dP_1\wedge dP_2)(e_1,e_2)&=(x_0+HS)(N_3+Hx_3),\\
(dQ_1\wedge dQ_2)(e_1,e_2)&=(c\,x_3-HN_3)(c\,S-Hx_0).
\end{split}
\end{equation}
Along $\partial M$,
\begin{equation}\label{eq:bdry}
\begin{split}
q'&=-\sn\rho\,n\,e_r,\\
P_3'&=\sn\rho\,(\cs\rho\,s+H\sn\rho),\\
\tilde q'&=-\sn\rho\,(\sn\rho\,s-cH\cs\rho)\,e_r ,
\end{split}
\end{equation}
where $e_r=(\cos\theta,\sin\theta)$ is identified with a vector of $\R^2$.
\end{lemma}

\begin{proof}
By Lemma~\ref{lem:wedge}, $dP_1\wedge dP_2=\det(x,N,T_1W,T_2W)$ and
$dQ_1\wedge dQ_2=\det(x,N,R_1W,R_2W)$, where, writing $W=\sum_kW_kE_k$,
\begin{equation*}
\begin{split}
T_iW&=W_0E_i-c\,W_iE_0,\\
 R_1W&=W_2E_3-W_3E_2,\\
  R_2W&=W_3E_1-W_1E_3 .
\end{split}
\end{equation*}
Write $e_j=\sum_k e_j^kE_k$ for $j=1,2$. Since $\langle E_0,e_j\rangle=c\,e_j^0$ and
$\langle E_k,e_j\rangle=e_j^k$ for $k\ge1$,
\begin{equation*}
\begin{split}
\langle T_iW,e_j\rangle&=W_0e_j^i-W_ie_j^0,\\
\langle R_1W,e_j\rangle&=W_2e_j^3-W_3e_j^2,\\
\langle R_2W,e_j\rangle&=W_3e_j^1-W_1e_j^3 .
\end{split}
\end{equation*}
Substituting in the first identity of Lemma~\ref{lem:det}, the terms in $W_1W_2$ cancel and
\begin{align*}
\det(x,N,T_1W,T_2W)&=W_0\Big[W_0\big(e_1^1e_2^2-e_1^2e_2^1\big)-W_1\big(e_1^0e_2^2-e_1^2e_2^0\big)
+W_2\big(e_1^0e_2^1-e_1^1e_2^0\big)\Big]\\
&=W_0\det(W,E_3,e_1,e_2),\\
\det(x,N,R_1W,R_2W)&=W_3\Big[W_1\big(e_1^2e_2^3-e_1^3e_2^2\big)-W_2\big(e_1^1e_2^3-e_1^3e_2^1\big)
+W_3\big(e_1^1e_2^2-e_1^2e_2^1\big)\Big]\\
&=W_3\det(E_0,W,e_1,e_2).
\end{align*}
 
 Since 
 $$W_0=c(x_0+HS),\qquad W_3=c\,x_3-HN_3,\qquad \langle W,x\rangle=1,\qquad \langle W,N\rangle=-H,$$
  the second identity gives \eqref{eq:wedges}. For \eqref{eq:bdry}, 
  $$P_i'=\sn\rho\,\omega_{T_i}(t)
=c\,\sn\rho\,\langle T_iW,N\times t\rangle=c\,\sn\rho\,\langle T_iW,\mu\rangle,$$
 and
similarly for $Q_j$. By $x=\cs\rho\,E_0+\sn\rho\,e_r$ and \eqref{eq:Nmu}, along $\partial M$
\begin{equation*}
\begin{split}
W&=c(\cs\rho+Hs\,\sn\rho)E_0+(c\,\sn\rho-Hs\,\cs\rho)\,e_r-Hn\,E_3,\\
\mu&=c\,n\,\sn\rho\,E_0-n\,\cs\rho\,e_r+s\,E_3,
\end{split}
\end{equation*}
and \eqref{eq:bdry} follows using $s^2+n^2=1$ and $\cs^2\rho+c\,\sn^2\rho=1$.
\end{proof}

\begin{lemma}[Flux formulas]\label{lem:flux}
The following hold.
\begin{enumerate}
\item[(a)] $\cs\rho\int_0^{2\pi}s\,d\theta=-2\pi H\sn\rho$.
\item[(b)] $\int_0^{2\pi}n\,e_r\,d\theta=0$ and $\int_0^{2\pi}s\,e_r\,d\theta=0$.
\item[(c)] $\int_M(x_0N_3+c\,x_3S)\dA=-\pi\sn^2\rho$.
\end{enumerate}
Consequently, if $\cs\rho\neq0$ the mean of $s$ is $\bar s=-H\ta\rho$ and $|H|\ta\rho\le1$.
If $\cs\rho=0$ (that is, $c=1$ and $\rho=\pi/2$), then $H=0$.
\end{lemma}

\begin{proof}
Each $\omega_B$ is closed, so $\oint_{\partial M}\omega_B=\int_M d\omega_B=0$. By
\eqref{eq:bdry} this gives (a) (from $T_3$), $\int n\,e_r=0$ (from $T_1,T_2$), and
$\int(\sn\rho\,s-cH\cs\rho)e_r=0$ (from $R_1,R_2$). Since $\int e_r=0$, (b) follows.

For (c), let $\xi=X_1\,dX_2-X_2\,dX_1$, restricted to $M$. Then $d\xi=2\,dX_1\wedge dX_2$ and
\[
(dX_1\wedge dX_2)(e_1,e_2)=X_1(e_1)X_2(e_2)-X_2(e_1)X_1(e_2)=\det(E_0,E_3,e_1,e_2),
\]
where the last equality follows by expanding along the first two columns, since the
permutation $(0,3,1,2)$ is even.   By Lemma~\ref{lem:det} with $w=E_0$, $w'=E_3$, and using
$c\langle E_0,x\rangle=x_0$, $\langle E_0,N\rangle=-S$, $\langle E_3,x\rangle=x_3$ and
$\langle E_3,N\rangle=N_3$,
\[
\det(E_0,E_3,e_1,e_2)=x_0N_3+c\,x_3S .
\]
Along $\partial M$ we have $\xi(\tfrac{d}{d\theta})=\sn^2\rho$. Since $\partial M$ is oriented
by $-t$, Stokes' theorem gives 
$$2\int_M(x_0N_3+c\,x_3S)\dA=\int_{\partial M}\xi=-2\pi\sn^2\rho.$$

The consequences follow from (a) and $|s|\le1$.
\end{proof}

\begin{remark} \label{rem:signH}
Along $\partial M$ the inward conormal is $-\mu=n\,\partial_r-s\,E_3$. If $x(M)$ is a closed
disk bounded by the circle in a totally umbilical surface, then, by rotational symmetry
about the axis of the circle, $s$ is constant, and by Lemma~\ref{lem:flux}(a) it equals
$-H\ta\rho$ (and $H=0$ if $\cs\rho=0$). If $H\neq0$, this disk lies on one side of the totally
geodesic surface $\M\cap\{X_3=0\}$. With the orientation $N=t\times\mu$, the one contained in
$\{X_3\ge0\}$ has $s<0$, hence $H>0$, and the one contained in $\{X_3\le0\}$ has $H<0$. This fixes
the sign of $H$ in Theorem~\ref{thm:A}. If $c=1$ and
$\rho=\pi/2$, the circle is a great circle, contained in infinitely many totally geodesic
spheres. Then $H=0$ and the disks of Theorem~\ref{thm:A} are the hemispheres of these spheres
bounded by the circle.
\end{remark}

For a smooth $f\colon\R/2\pi\mathbb Z\to\R$ write
\begin{equation*}
\begin{split}
f&=\bar f+\sum_{m\ge1}(a_m\cos m\theta+b_m\sin m\theta),\\
\mathcal E_m(f)&=\pi(a_m^2+b_m^2),
\end{split}
\end{equation*}
so that $\int_0^{2\pi}(f-\bar f)^2=\sum_{m\ge1}\mathcal E_m(f)$.

\begin{lemma}\label{lem:radial}
Let $f\colon\R/2\pi\mathbb Z\to\R$ and $\gamma\colon\R/2\pi\mathbb Z\to\R^2$ be smooth with
$\gamma'=f\,e_r$. Then $\int_0^{2\pi}f\,e_r\,d\theta=0$, so $\mathcal E_1(f)=0$, and
\[
\int_0^{2\pi}|\gamma'|^2d\theta-\int_0^{2\pi}\det(\gamma,\gamma')\,d\theta=\sum_{m\ge2}\frac{m^2}{m^2-1}\,\mathcal E_m(f).
\]
In particular the left hand side is nonnegative, and it vanishes if and only if $f$ is constant.
\end{lemma}

\begin{proof}
Since $\gamma$ is closed, $\int_0^{2\pi}f\,e_r\,d\theta=\int_0^{2\pi}\gamma'\,d\theta=0$, that is,
$a_1=b_1=0$. Identify $\R^2=\mathbb C$, $e_r=e^{i\theta}$, and write
$$\gamma=\sum_k\gamma_ke^{ik\theta}$$
$$f=\sum_m\varphi_me^{im\theta},$$
with
$\varphi_{-m}=\overline{\varphi_m}$ and $|\varphi_m|^2=\frac14(a_m^2+b_m^2)$ for $m\ge1$.
From $\gamma'=fe^{i\theta}$ we get $ik\,\gamma_k=\varphi_{k-1}$. Parseval's identity, as in
Hurwitz's proof \cite{Hur01}, gives
\[
\int|\gamma'|^2-\int\det(\gamma,\gamma')=2\pi\sum_kk(k-1)|\gamma_k|^2=2\pi\sum_{m\neq-1}\frac{m}{m+1}|\varphi_m|^2 .
\]
Here $m=k-1$, and $m=-1$ does not occur because the term $k=0$ vanishes.
The terms $m=0,1$ vanish. For $m\ge2$, grouping $\pm m$ gives
$\frac{m}{m+1}+\frac{m}{m-1}=\frac{2m^2}{m^2-1}$, and
$2\pi\,\frac{2m^2}{m^2-1}|\varphi_m|^2=\frac{m^2}{m^2-1}\mathcal E_m(f)$.
\end{proof}

\section{Proof of Theorem~\ref{thm:B}}\label{sec:B}

Throughout this section we assume the hypotheses of Theorem~\ref{thm:B} and use the notation
of Section~\ref{sec:prelim}. Since $E_3$ is the unit normal of $\M\cap\{X_3=0\}$, we must show
that $n=\langle N,E_3\rangle$ is constant along $\partial M$ and that $A=H\,\mathrm{Id}$ there.

By \eqref{eq:bdry}, the boundary curves $q$ and $\tilde q$ satisfy
\begin{equation*}
\begin{split}
&q'=f_P\,e_r,\qquad \tilde q'=f_Q\,e_r,\\
&f_P:=-\sn\rho\,n,\qquad f_Q:=-\sn\rho\,(\sn\rho\,s-cH\cs\rho),
\end{split}
\end{equation*}
and since $q$ and $\tilde q$ are closed curves, Lemma~\ref{lem:radial} applies to $\gamma=q$
and to $\gamma=\tilde q$. Let
\begin{equation*}
\begin{split}
\mathcal A:&=\int_0^{2\pi}|q'|^2d\theta-\int_0^{2\pi}\det(q,q')\,d\theta,\\
\mathcal C:&=\int_0^{2\pi}|\tilde q'|^2d\theta-\int_0^{2\pi}\det(\tilde q,\tilde q')\,d\theta .
\end{split}
\end{equation*}
Adding a constant to a function does not change $\mathcal E_m$ for $m\ge1$, so
$\mathcal E_m(f_P)=\sn^2\rho\,\mathcal E_m(n)$ and $\mathcal E_m(f_Q)=\sn^4\rho\,\mathcal E_m(s)$.
Lemma~\ref{lem:radial} therefore gives
\begin{equation}\label{eq:AC}
\begin{split}
\mathcal A&=\sn^2\rho\sum_{m\ge2}\dfrac{m^2}{m^2-1}\mathcal E_m(n)\ \ge0,\\
\mathcal C&=\sn^4\rho\sum_{m\ge2}\dfrac{m^2}{m^2-1}\mathcal E_m(s)\ \ge0 .
\end{split}
\end{equation}

\begin{proposition}\label{prop:key}
$\displaystyle \mathcal A+c\,\mathcal C=-\int_0^{2\pi}(P_3')^2\,d\theta
=-\sn^2\rho\int_0^{2\pi}(\cs\rho\,s+H\sn\rho)^2d\theta .$
\end{proposition}

\begin{proof}
By \eqref{eq:stokes} applied to $F=q$ and to $F=\tilde q$,
\[
-\int_0^{2\pi}\det(q,q')\,d\theta=2\int_M dP_1\wedge dP_2,
\]
\[
-\int_0^{2\pi}\det(\tilde q,\tilde q')\,d\theta=2\int_M dQ_1\wedge dQ_2 .
\]
Therefore $\mathcal A+c\,\mathcal C=\mathcal B+2\,\mathcal I$, where
\begin{equation*}
\begin{split}
\mathcal B&:=\int_0^{2\pi}\big(|q'|^2+c\,|\tilde q'|^2\big)\,d\theta,\\
\mathcal I&:=\int_M\big(dP_1\wedge dP_2+c\,dQ_1\wedge dQ_2\big)(e_1,e_2)\dA .
\end{split}
\end{equation*}

  By \eqref{eq:wedges} and $c^2=1$,
\[
(x_0+HS)(N_3+Hx_3)+c\,(c\,x_3-HN_3)(c\,S-Hx_0)=(1+cH^2)(x_0N_3+c\,x_3S),
\]
so Lemma~\ref{lem:flux}(c) gives $\mathcal I=-\pi(1+cH^2)\sn^2\rho$.

We now consider the boundary term $\mathcal B$.  By \eqref{eq:bdry} and $n^2=1-s^2$,
$|q'|^2=\sn^2\rho\,(1-s^2)$ and $|\tilde q'|^2=\sn^2\rho\,(\sn\rho\,s-cH\cs\rho)^2$. Hence,
writing $\int$ for $\int_0^{2\pi}\cdots d\theta$,
\[
\mathcal B=\sn^2\rho\Big[2\pi-\int s^2+c\,\sn^2\rho\int s^2-2H\sn\rho\cs\rho\int s
+2\pi cH^2\cs^2\rho\Big].
\]
Adding $\mathcal B$ and $2\mathcal I$, and using $1-c\,\sn^2\rho=\cs^2\rho$, 
$c(\cs^2\rho-1)=-\sn^2\rho$ and \eqref{eq:bdry}, we obtain
\begin{equation*}
\begin{split}
\mathcal A+c\,\mathcal C&=-\sn^2\rho\Big[\cs^2\rho\int s^2+2H\sn\rho\cs\rho\int s
+2\pi H^2\sn^2\rho\Big]\\
&=-\sn^2\rho\int_0^{2\pi}(\cs\rho\,s+H\sn\rho)^2\,d\theta\\
&=-\int_0^{2\pi}(P_3')^2\,d\theta.
\end{split}
\end{equation*}
 
\end{proof}

\begin{remark}\label{rem:MN}
The proof of Proposition~\ref{prop:key} does not use Lemma~\ref{lem:flux}(a). Formally, for
$c=0$ and after normalizing the circle to have radius one, the term $\mathcal C$ drops out and the identity becomes $\mathcal A=-\int(P_3')^2$,
which is the comparison between signed area and energy of the force curve in \cite{MN26}.
The new term $c\,\mathcal C$ has the favourable sign for $c=1$ and the unfavourable one for
$c=-1$.
\end{remark}

\begin{remark}\label{rem:genus}
As in \cite[Remark~1.1]{MN26}, the disk hypothesis is used only to obtain the primitives
$P_i,Q_j$. Identities \eqref{eq:wedges}, Lemma~\ref{lem:flux} and the value of $\mathcal I$
hold for compact surfaces of any genus whose boundary is mapped diffeomorphically onto the
circle, but in positive genus the forms $\omega_B$ need not be exact.
\end{remark}

\begin{proposition}\label{prop:bdry}
The functions $s$ and $n$ are constant along $\partial M$, and $\cs\rho\,s+H\sn\rho=0$.
\end{proposition}

\begin{proof}
\emph{The case $c=1$.} By Proposition~\ref{prop:key} and \eqref{eq:AC},
$0\le\mathcal A+\mathcal C=-\int(P_3')^2\le0$. Hence $\mathcal A=\mathcal C=0$ and
$\cs\rho\,s+H\sn\rho\equiv0$. By \eqref{eq:AC}, $\mathcal E_m(s)=\mathcal E_m(n)=0$ for
$m\ge2$. Together with Lemma~\ref{lem:flux}(b), $s$ and $n$ are constant. This case does
not use Lemma~\ref{lem:flux}(a).

\emph{The case $c=-1$ and $|H|>1$.} It suffices to show that $n$ is constant. Indeed, then
$s^2=1-n^2$ is constant, and $s$ is continuous on the connected curve $\partial M$, so $s$ is
constant, equal to its mean $-H\tanh\rho$.

By Lemma~\ref{lem:flux}, $\cosh\rho\,s+H\sinh\rho=\cosh\rho\,(s-\bar s)$ and $\mathcal E_1(s)=0$. Let
\[
D:=\int_0^{2\pi}(s-\bar s)^2d\theta=\sum_{m\ge2}\mathcal E_m(s).
\]
Then Proposition~\ref{prop:key} reads $\mathcal A-\mathcal C=-\sinh^2\rho\cosh^2\rho\,D$.
Inserting \eqref{eq:AC},  
we get
\begin{equation}\label{eq:modal}
\sum_{m\ge2}\Big[\frac{m^2}{m^2-1}\,\mathcal E_m(n)+\Big(1-\frac{\sinh^2\rho}{m^2-1}\Big)\mathcal E_m(s)\Big]=0 .
\end{equation}
The coefficients of $\mathcal E_m(n)$ are positive, but those of $\mathcal E_m(s)$ are
negative for the low modes with $m^2-1<\sinh^2\rho$.

Let $v=(-s,n)\colon\R/2\pi\mathbb Z\to\mathbb S^1$ and
$\bar v=\frac1{2\pi}\int v\,d\theta=\sigma(\cos\alpha,\sin\alpha)$, $0\le\sigma\le1$. Then
$$D+\int(n-\bar n)^2=\int|v-\bar v|^2=2\pi(1-\sigma^2),$$
$$ \sigma\cos\alpha=-\bar s=H\tanh\rho.$$
If $\sigma=1$, then $v$ is constant and so is $n$. Suppose, to the contrary, that $\sigma<1$, and let
$\kappa=\cos^2\alpha$. Then $\kappa>0$, since $\sigma\cos\alpha=H\tanh\rho\ne0$, and
$\sigma^2\kappa\le\sigma^2<1$.

\begin{enumerate}
\item\label{st:1} Since $|H|>1$, $\tanh^2\rho<H^2\tanh^2\rho=\sigma^2\kappa<1$. This gives
$\sinh^2\rho<\frac{\sigma^2\kappa}{1-\sigma^2\kappa}$.

\item Let $R=\sum_{m\ge2}\frac{\mathcal E_m(s)}{m^2-1}$. By \eqref{eq:modal},
$\sum_{m\ge2}\frac{m^2}{m^2-1}\mathcal E_m(n)=\sinh^2\rho\,R-D$. Together with
$\frac{m^2}{m^2-1}\ge1$ and $\mathcal E_1(n)=0$, this gives
\begin{equation}\label{eq:star}
2\pi(1-\sigma^2)=D+\int(n-\bar n)^2\le D+\sum_{m\ge2}\dfrac{m^2}{m^2-1}\mathcal E_m(n)=\sinh^2\rho\,R .
\end{equation}

\item\label{st:3} Let 
$$y=\langle v,(\cos\alpha,\sin\alpha)\rangle,\qquad z=\langle v,(-\sin\alpha,\cos\alpha)\rangle.$$
 Then $y^2+z^2=1$, $\bar y=\sigma$,
$\bar z=0$ and $s=-\cos\alpha\,y+\sin\alpha\,z$. Since $\mathcal E_m$ is a nonnegative
quadratic form, for each $m\ge2$,
$$\mathcal E_m(s)\le2\kappa\,\mathcal E_m(y)+2(1-\kappa)\mathcal E_m(z).$$

Since $1-y\ge0$ has integral $2\pi(1-\sigma)$, every complex Fourier coefficient
$\frac1{2\pi}\int y\,e^{-im\theta}\,d\theta$ of $y$ with $m\neq0$ has modulus at most
$1-\sigma$. So $\mathcal E_m(y)\le4\pi(1-\sigma)^2$, and with
$\sum_{m\ge2}\frac1{m^2-1}=\frac34$,
\[
\sum_{m\ge2}\dfrac{\mathcal E_m(y)}{m^2-1}\le3\pi(1-\sigma)^2 .
\]
Also
\[
\sum_{m\ge2}\dfrac{\mathcal E_m(z)}{m^2-1}\le\frac13\int z^2=\frac13\int(1-y)(1+y)\le\frac{4\pi}{3}(1-\sigma).
\]
Hence, $R\le2\pi(1-\sigma)\Lambda$, where $ \Lambda=3\kappa(1-\sigma)+\tfrac43(1-\kappa)$.

\item\label{st:4} We claim $\sigma^2\kappa\Lambda\le(1+\sigma)(1-\sigma^2\kappa)$ for $(\sigma,\kappa)\in[0,1]^2$. Let
\[
e(\kappa)=3\big[(1+\sigma)(1-\sigma^2\kappa)-\sigma^2\kappa\Lambda\big].
\]
Then $e(0)>0$ and $e(1)=3(1-\sigma)(1+2\sigma-2\sigma^2)\ge0$. If $9\sigma\le5$, $e$ is concave in $\kappa$.
If $9\sigma >5$, its vertex is at $\frac{3\sigma+7}{2(9\sigma-5)}\ge\frac54$. In both cases $e\ge0$ on $[0,1]$.
\end{enumerate}
If $R>0$, steps \ref{st:1}, \ref{st:3} and \ref{st:4} give
\[
\sinh^2\rho\,R<\frac{\sigma^2\kappa}{1-\sigma^2\kappa}R\le2\pi(1-\sigma)\frac{\sigma^2\kappa\Lambda}{1-\sigma^2\kappa}\le2\pi(1-\sigma^2),
\]
which contradicts \eqref{eq:star}. If $R=0$, then \eqref{eq:star} gives $\sigma=1$, again a
contradiction. Hence $\sigma=1$ and $n$ is constant.
\end{proof}

\begin{remark}\label{rem:largeH}
If $c=-1$ and $|H|\ge2/\sqrt3$, identity \eqref{eq:modal} gives the result directly. Indeed,
$|H|\tanh\rho\le1$ gives $\sinh^2\rho\le1/(H^2-1)\le3\le m^2-1$ for every $m\ge2$, so all the
coefficients in \eqref{eq:modal} are nonnegative. Hence $\mathcal E_m(n)=0$ for $m\ge2$ and, by
Lemma~\ref{lem:flux}(b), $n$ is constant. The Fourier estimate is needed only for
$1<|H|<2/\sqrt3$, where the mode $m=2$ may have a negative coefficient.
\end{remark}

\begin{corollary}\label{cor:umb}
$A=H\,\mathrm{Id}$ along $\partial M$.
\end{corollary}

\begin{proof}
If $\cs\rho\ne0$, then $s=-H\ta\rho$ and, by \eqref{eq:curvatures}, $A(t,t)=-s\ct\rho=H$.
If $\cs\rho=0$, then $H=0=A(t,t)$.

Since $s,n$ are constant and $\frac{d}{d\ell}\partial_r=\ct\rho\,t$, we get
$dN(t)=s\ct\rho\,t$, so $At=-s\ct\rho\,t=Ht$. As $A$ is self-adjoint with trace $2H$,
also $A\mu=H\mu$.
\end{proof}

Proposition~\ref{prop:bdry} and Corollary~\ref{cor:umb} prove Theorem~\ref{thm:B}.

%%%%%%%%%%%%%
\section{Proof of Theorem~\ref{thm:A}}\label{sec:A}

If $c=-1$ and $|H|\le1$, Theorem~\ref{thm:A} is proved in \cite{lo}. We therefore assume
that $c=1$, or that $c=-1$ and $|H|>1$, so that Theorem~\ref{thm:B} applies.

In a conformal coordinate $z$, the Hopf differential $Q=\langle\nabla_{\partial_z}\partial_z x,N\rangle dz^2$
is holomorphic, since $H$ is constant and the Codazzi equation holds in $\M$ \cite{Hop83}.
By Theorem~\ref{thm:B} it vanishes along $\partial M$. As in \cite[\S4.1]{MN26},
using a conformal boundary chart and Schwarz reflection, we get $Q\equiv0$, that is,
$A=H\,\mathrm{Id}$ on $M$.

Then $V:=N+Hx$ satisfies $dV=-dx\circ A+H\,dx=0$, so $V$ is a constant vector of $\R^4$
and $\langle x,V\rangle=cH$. Thus $x(M)$ lies in the totally umbilical surface
\[
\Sigma=\M\cap\{\langle X,V\rangle=cH\},\qquad \langle V,V\rangle=1+cH^2,
\]
which contains the circle and has constant Gauss curvature $c+H^2>0$. If $c=1$, $\Sigma$
is the intersection of $\mathbb S^3$ with an affine hyperplane, hence a round $2$-sphere.
If $c=-1$ and $|H|>1$, then $\langle V,V\rangle=1-H^2<0$, so $V$ is timelike and $\Sigma$
is a geodesic sphere of $\mathbb H^3$. Moreover $x\colon M\to\Sigma$ is a local isometry.

We use the following topological lemma.

\begin{lemma}[{\cite[Lemma 4.1]{MN26}}]\label{lem:degree}
Let $F\colon M\to\R^2$ be a smooth immersion of a compact connected oriented surface with
boundary such that $F|_{\partial M}$ is a diffeomorphism onto a Jordan curve $\Gamma$.
Then $F$ is an embedding onto the closed region bounded by $\Gamma$.
\end{lemma}

With the boundary oriented by $-t$ and inward conormal $-\mu$,
\eqref{eq:curvatures} gives geodesic curvature $k_g=-n\ct\rho$. Hence
$$\int_{\partial M}k_g\,d\ell=-2\pi n\cs\rho,$$
 and Gauss--Bonnet gives
\[
\mathrm{Area}(M)=\frac{2\pi(1+n\cs\rho)}{c+H^2}.
\]
We claim $n\cs\rho<1$. For $c=1$, $|n|\le1$ and $0\le\cos\rho<1$. For $c=-1$,
$n^2\cosh^2\rho=\cosh^2\rho-H^2\sinh^2\rho<1$ because $|H|>1$.
Hence $\mathrm{Area}(M)<4\pi/(c+H^2)=\mathrm{Area}(\Sigma)$. So $x(M)\ne\Sigma$, since
otherwise the area formula for the local isometry $x$ would give
$\mathrm{Area}(M)\ge\mathrm{Area}(\Sigma)$. Stereographic projection of $\Sigma$ from a
point of $\Sigma\setminus x(M)$ and Lemma~\ref{lem:degree} show that $x$ is an embedding
onto one of the two closed caps of $\Sigma$ bounded by the circle. \qed

\section*{Acknowledgements}
The author thanks Davi Maximo for reading a preliminary version of these computations and for his encouragement to write them up.
The author has been partially supported by MINECO/MICINN/FEDER grant no. PID2023-150727NB-I00, and by the ``Mar\'{\i}a de Maeztu'' Excellence Unit IMAG, reference CEX2020-001105-M, funded by MCINN/AEI/10.13039/501100011033/CEX2020-001105-M.
%%%%%%%%%%%%%%%%%%%%%%%%%%%%%%%%%%%%%%%%%%%%%%%%%%%%%%%%%%%%%%

\end{document}